\documentclass[reqno,a4paper,11pt]{amsart}
\usepackage[ngerman, english]{babel}
\usepackage[all,poly]{xy}
\usepackage[mathcal]{eucal}
\usepackage{enumerate}
\usepackage{amsmath,amssymb, amsthm, amsfonts}
\usepackage[pagebackref,colorlinks]{hyperref}
\usepackage{graphicx}
\usepackage{mathabx}

\theoremstyle{plain}
\newtheorem {lemma}{Lemma}
\newtheorem {proposition}[lemma]{Proposition}
\newtheorem {theorem}[lemma]{Theorem}
\newtheorem {SCT}[lemma]{Sandwich Classification Theorem}
\newtheorem {corollary}[lemma]{Corollary}

\theoremstyle{definition}
\newtheorem{definition}[lemma]{Definition}

\newtheorem*{remark}{Remark}

\newcommand{\N}{\mathbb{N}}
\DeclareMathOperator*{\plus}{\overset{.}{\bigplus}}
\newcommand{\+}{\overset{.}{+}}
\newcommand{\minus}{\overset{.}{-}}
\newcommand{\h}{\mathfrak{H}}
\newcommand{\hp}{\leftharpoonup}
\newcommand{\tr}{\operatorname{tr}}
\newcommand{\K}{\mathfrak{L}}
\newcommand{\M}{\mathfrak{M}}

\newcommand{\GL}{\operatorname{GL}}

\newcommand{\E}{\operatorname{E}}
\newcommand{\C}{\operatorname{C}}
\newcommand{\U}{\operatorname{U}}
\newcommand{\EU}{\operatorname{EU}}

\newcommand{\NU}{\operatorname{NU}}
\newcommand{\CU}{\operatorname{CU}}

\newcommand{\Center}{\operatorname{Center}}

\newcommand{\Mat}{\operatorname{M}}

\newcommand{\ind}{\operatorname{ind}}
\newcommand{\V}{\mathcal{V}}
\newcommand{\Hy}{\mathcal{H}}
\newcommand{\I}{\mathfrak{I}}
\newcommand{\J}{\mathfrak{J}}
\DeclareFontEncoding{LS1}{}{}
\DeclareFontSubstitution{LS1}{stix}{m}{n}
\DeclareSymbolFont{symbols2}{LS1}{stixfrak} {m} {n}
\DeclareMathSymbol{\operp}{\mathbin}{symbols2}{"A8}

\title[The E-normal structure of Petrov's odd unitary groups]{The E-normal structure of Petrov's odd unitary groups over commutative rings}
\author{Raimund Preusser}
\address{Department of Mathematics,
University of Brasilia, Brazil}
\email{raimund.preusser@gmx.de}
\subjclass[2010]{20G35, 20H25} 
\keywords{odd unitary groups, E-normal subgroups, sandwich classification}

\begin{document}

\begin{abstract}
For an odd quadratic space $\V$ of Witt index $\geq 3$ over a commutative ring with pseudoinvolution, we classify the subgroups of the odd unitary group $\U(\V)$ that are normalized by the elementary subgroup $\EU_{(e_1,e_{-1})}(\V)$ defined by a hyperbolic pair $(e_1,e_{-1})$ in $\V$. Further we correct some minor mistakes that exist in the literature on odd unitary groups.
\end{abstract}

\maketitle

\section{Introduction}
In the second half of the 19th century mathematicians became interested in the normal subgroups of classical groups, at first over finite fields and then over infinite fields. The case of finite fields was of special interest because normal subgroups here yielded concrete examples of finite nonabelian simple groups, which were important for Galois theory. In 1870, C. Jordan \cite{jordan} showed that the normal subgroups of the general linear group $\GL_n(\mathbb{F}_p)$, where $\mathbb{F}_p$ is a finite field of prime order $p$, are precisely the overgroups of the elementary subgroup $\E_n(\mathbb{F}_p)$ and the subgroups of $\Center(\GL_n(\mathbb{F}_p))$. In 1900, L. E. Dickson \cite{dickson, dickson_book} extended Jordan's result to all fields and further to all classical groups. In the period 1948-1963, J. Dieudonne \cite{dieudonne_book_1,dieudonne_1,dieudonne_2,dieudonne_book_2} generalized Dickson's results to division rings. 

In 1961, W. Klingenberg \cite{klingenberg_1} proved that if $R$ is a commutative local ring, $n\geq 3$, and $H$ is a subgroup of $\GL_n(R)$, then 
\begin{equation}
H\text{ is normal}~\Leftrightarrow~ \exists \text{ ideal }I:~\E_n(R,I)\subseteq H\subseteq \C_n(R,I)\end{equation}
where $E_n(R,I)$ denotes the elementary subgroup of level $I$ and $\C_n(R,I)$ the full congruence subgroup of level $I$. If $R$ is a field or division ring one gets back the results of Dickson resp. Dieudonne. Klingenberg proved analog results for the symplectic and the orthogonal groups \cite{klingenberg_2,klingenberg_3}. In \cite{klingenberg_4}, Klingenberg tried to generalize (1) to arbitrary (not necessarily commutative) local rings $R$. 

In 1964, H. Bass \cite{bass} noticed that Klingenberg had made a mistake. What Klingenberg really had shown in \cite{klingenberg_4} was that if $R$ is a local ring, $n\geq 3$, and $H$ is a subgroup of $\GL_n(R)$, then
\begin{equation}
H\text{ is E-normal}~\Leftrightarrow~ \exists \text{ ideal }I:~\E_n(R,I)\subseteq H\subseteq \C_n(R,I).\end{equation}
Recall that a subgroup of a classical group $G$ is called {\it E-normal} if it is normalized by the elementary subgroup of $G$. Further Bass showed that (2) holds true if $R$ is an arbitrary ring and $n$ is large enough with respect to the stable rank of $R$, including the case that $R$ is semilocal and $n\geq 3$. From that moment on the focus of research changed. It was now on the E-normal structure of classical and classical-like groups instead of the normal structure. Amazingly, it seems that today it is still unknown what the normal subgroups of the general linear group over a commutative ring are (while a description of the E-normal subgroups is available, see the next paragraph).    

In 1972, J. S. Wilson \cite{wilson} proved that (2) holds if $R$ is commutative and $n\geq 4$. In 1973,  I. Z. Golubchik \cite{golubchik} added the case $n=3$. In 1981, L. Vaserstein \cite{vaserstein} showed that (2) holds if $R$ is almost commutative (i.e. module finite over its center) and $n\geq 3$.

A natural question is if similar results hold for other classical groups. In 1969, A.
Bak \cite{bak_2} defined hyperbolic unitary groups $\U_{2n}(R,\Lambda)$ which include as special cases the classical Chevalley groups of type $C_m$ and $D_m$. He proved that if $H$ is a subgroup of $\U_{2n}(R,\Lambda)$, then
\begin{align}
H \text{ is E-normal}~\Leftrightarrow~ \exists\text{ form ideal } (I,\Gamma):\EU_{2n}((R,\Lambda),(I,\Gamma)) \subseteq H\subseteq \CU_{2n}((R,\Lambda),(I,\Gamma))
\end{align}
if $n$ is large enough with respect to the Bass-Serre dimension of $R$. In 1995, L. N. Vaserstein and H. You \cite{vaserstein-you} gave an incorrect proof of (3) for the case that $R$ is almost commutative and $n\geq 3$. The proof can be repaired when $2$ is invertible in $R$. In 2012, H. You \cite{you} proved that (3) holds if $R$ is commutative and $n\geq 4$ and one year later H. You and X. Zhou \cite{you-zhou} added the case that $n=3$. In 2017, the author \cite{preusser} proved that (3) holds if $R$ is almost commmutative and $n\geq 3$.

In 2005, V. A. Petrov \cite{petrov} defined odd unitary groups $\U(\V)$. These groups are very general, they include as special cases all the classical Chevalley groups, Bak's hyperbolic unitary groups and the Hermitian groups of G. Tang \cite{tang}. Recently Bak and the author \cite{bak-preusser} studied a subclass of Petrov's odd unitary groups, big enough to contain all the classical Chevalley groups, and classified the E-normal subgroups of the members of this subclass. Also recently W. Yu, Y. Li and H. Liu \cite{yu-li-liu} studied the E-normal subgroups of odd unitary groups under a $\Lambda$-stable rank condition. 

In this paper we investigate the E-normal subgroups of Petrov's odd unitary groups. The main result is the following: Let $\V$ be an odd quadratic space of Witt index $\geq 3$ over a commutative ring with pseudoinvolution. Choose three mutually orthogonal hyperbolic pairs $(e_1,e_{-1}),(e_2,e_{-2}),(e_3,e_{-3})$ in $\V$ and let $\V_0$ be the orthogonal complement of the odd quadratic space spanned by $e_1,e_2,e_3,e_{-3},e_{-2},e_{-1}$. Then the odd unitary group $\U(\V)$ is isomorphic to the odd hyperbolic unitary group $\U_{6}(\V_0)$ see \cite[p. 4756]{petrov}. We prove that if $H$ is a subgroup of $\U_{6}(\V_0)$, then
\begin{equation}
H \text{ is E-normal}~\Leftrightarrow~\exists\text{ odd form ideal } \I_0:\EU_{6}(\V_0,\I_0)\subseteq H\subseteq \CU_{6}(\V_0,\I_0).
\end{equation}
A priori, the notion of E-normality depends on the choice of the hyperbolic pairs $(e_1,e_{-1}),(e_2,e_{-2}),(e_3,e_{-3})$. Petrov has shown that it depends only on the choice of one hyperbolic pair, see \cite[Section 7]{petrov}. 

Furthermore we correct a couple of small mistakes that have been made in \cite{petrov} and later were  repeated in \cite{yu-tang} and \cite{yu-li-liu}. See the remarks in Sections 3 and 4.

The rest of the paper is organized as follows. In Section 2 we recall some standard notation
which will be used throughout the paper. In Section 3 we recall the definitions of Petrov's odd unitary groups and some important subgroups. In Section 4 we recall the definitions of the odd hyperbolic unitary groups and some important subgroups. In Section 5 we prove (4).

\section{Notation}
$\N$ denotes the set of positive integers and $\N_0$ the set of nonnegative integers. If $G$ is a group and $g,h\in G$, we let $^hg:=hgh^{-1}$ and $[g,h]:=ghg^{-1}h^{-1}$. By a ring we will always mean an associative ring with $1$ such that $1\neq 0$. Ideal will mean two-sided ideal. If $x_1,\dots,x_n$ are elements of a ring $R$, then we denote by $I(x_1,\dots,x_n)$ the ideal of $R$ generated by the $x_i$'s. The set of all $n\times n$ matrices over a ring $R$ is denoted by $\Mat_n(R)$.

\section{Petrov's odd unitary groups}
\subsection{Odd unitary groups}
Throughout Section 3, $R$ denotes a ring and $~\bar{}~$ a pseudoinvolution on $R$. Recall that a {\it pseudoinvolution} on $R$ is an additive map $~\bar{}: R\rightarrow R$, $x\mapsto \bar x$ such that $\overline{xy}=\bar y \bar 1^{-1}\bar x$ and $\bar{\bar{x}}=x$ for any $x,y\in R$.
\begin{definition}
Let $V$ be a right $R$-module. A biadditive map $B:V\times V\rightarrow R$ such that $B(vx,wy)=\bar x \bar 1^{-1}B(v,w) y$ for any $v,w\in V$ and $x,y\in R$ is called a {\it sesquilinear form on $V$}. A sesquilinear form $B$ on $V$ is called an {\it anti-Hermitian form on $V$} if $B(v,w)=-\overline{B(w,v)}$ for any $v,w\in V$. If $B$ is an anti-Hermitian form on $V$, then the pair $(V,B)$ is called an {\it anti-Hermitian space}.
\end{definition}

\begin{definition}
Let $(V,B)$ be an anti-Hermitian space. The {\it Heisenberg group $\h$ of $(V,B)$} is defined as follows. As a set $\h=V\times R$. The composition is given by
\[(v,x)\+(w,y)=(v+w,x+y+B(v,w)).\]
The inverse $\minus (v,x)$ of an element $(v,x)\in\h$ is given by $\minus(v,x)=(-v,-x+B(v,v))$.
\end{definition}

\begin{definition}
Let $(V,B)$ be an anti-Hermitian space. Define a right action $\hp$ of the multiplicative monoid $(R,\cdot)$ on $\h$  by 
\[(v,x)\hp y=(vy,\bar y\bar 1^{-1}xy).\]
This action has the property that $(\alpha\+\beta)\hp y=\alpha\hp y\+ \beta\hp y$ for any $\alpha, \beta\in \h$ and $y\in R$. \\
\end{definition}

\begin{definition}
Let $(V,B)$ be an anti-Hermitian space. The {\it trace map} $\tr:\h\rightarrow R$ is defined by 
\[\tr(v,x)=x-\bar x-B(v,v).\]
One checks easily that $\tr$ is a group homomorphism from $\h$ to $(R,+)$.
\end{definition}

\begin{definition}
Let $(V,B)$ be an anti-Hermitian space. Set 
\begin{align*}
\K_{\min}&:=\{(0,x+\bar x)\mid x\in R\},\\
\K_{\max}&:=\ker\tr.
\end{align*}
A subgroup $\K$ of $\h$ that is stable under $\hp$ and lies between $\K_{\min}$ and $\K_{\max}$ is called an {\it odd form parameter}. Since $\K_{\min}$ and $\K_{\max}$ are subgroups of $\h$ that are stable under $\hp$, they are the smallest resp. the largest odd form parameter. 
\end{definition}

One checks easily that $\K_{\min}$ contains the commutator subgroup of $\h$. Hence any odd form parameter is a normal subgroup of $\h$.
\begin{definition}\label{defoqs}
Let $(V,B)$ be an anti-Hermitian space. Define the map 
\begin{align*}
Q:V&\rightarrow \h\\
v&\mapsto (v,0)
\end{align*}
Obviously
\begin{align*}
&Q(vx)=Q(v)\hp x~\forall v\in V, x\in R\text{ and}\\
&Q(v+w)=Q(v)\+Q(w)\minus(0,B(v,w))~\forall v,w\in V.
\end{align*}
If $\K$ is an odd form parameter, then the map $Q_\K:V\rightarrow \h/\K$ induced by $Q$ is called the {\it odd quadratic form on $V$ defined by $\K$}. The triple $\V:=(V, B, Q_\K)$ is called an {\it odd quadratic space}. 
\end{definition}

\begin{remark}
In \cite{petrov} the odd quadratic form $Q_\K$ is not explicitly mentioned, but it appears implicitly in the definition of an odd unitary group. Petrov calls the pair $q=(B,\K)$ an odd quadratic form and the pair $(V,q)$ and odd quadratic space.
\end{remark}

Until the end of Section 4, $\V=(V, B, Q_\K)$ denotes an odd quadratic space. 
\begin{definition}
The group 
\[\U(\V)=\{\sigma\in \GL(V)\mid B(\sigma v,\sigma w)=B(v,w)\land Q_\K(\sigma v)=Q_\K(v)~\forall v,w\in V\}\]
is called the {\it odd unitary group of $\V$}.
\end{definition}

In the following we will denote the inverse of an element $\sigma\in \U(\V)$ by $\widetilde \sigma$ instead of $\sigma^{-1}$ in order to get nicer-looking formulas.
 
\subsection{Eichler-Siegel-Dickson transvections}
\begin{definition}
Two elements $u,v\in V$ are called {\it orthogonal} if $B(u,v)=0$. An element $w\in V$ is called {\it isotropic} if $Q_\K(w)=0$. 
\end{definition}
\begin{definition}
Let $u\in V$ be an isotropic element and $(v,x)\in \K$ such that $u$ and $v$ are orthogonal. Then the endomorphism $T_{uv}(a)$ of $V$ mapping
\[w\mapsto w+u\bar 1^{-1}(B(v,w)+xB(u,w))+vB(u,w)\]
is called an {\it Eichler-Siegel-Dickson transvection} or an {\it ESD transvection}.
\end{definition}
\begin{lemma}\label{lemesd}
The ESD transvections lie in $\U(\V)$. Furthermore
\begin{align*}
&T_{uv_1}(x)T_{uv_2}(y)=T_{u,v_1+v_2}(x+y+B(v_1,v_2))\text{ and}\\
&^\sigma\!T_{u,v}(x)=T_{\sigma u,\sigma v}(x)\text{ for any }\sigma \in \U(\V).
\end{align*}
\end{lemma}
\begin{proof}
See \cite[Section 4]{petrov}.
\end{proof}
\subsection{Principal congruence subgroups}
The submodule $V^{ev}:=\{v\in V\mid (v,x)\in \K\text{ for some }x\in R\}$ of $V$ is called the {\it even part of $V$}. We denote the Heisenberg group of the anti-Hermitian space $(V^{ev}, B|_{V^{ev}})$ by $\h^{ev}$. It is easy to see that any $\sigma \in \U(\V)$ restricts to an automorphism of $V^{ev}$ (in fact any $\sigma \in \U(\V)$ has the stronger property that $(\sigma v,x)\in \K$ for any $(v,x)\in\K$).
\begin{definition}\label{defrelfp}
Let $I$ be an ideal of $R$ such that $\bar I=I$. Set 
\begin{align*}
\M_{\min}&:=\K\hp I\+\{(0,x+\bar x)\mid x\in I\},\\
\M_{\max}&:=\{(v,x)\in \K\mid x\in I, B(v,w)\in I~\forall w\in V^{ev}\}.
\end{align*}
A subgroup $\M$ of $\h$ that is stable under $\hp$ and lies between $\M_{\min}$ and $\M_{\max}$ is called an {\it relative odd form parameter of level $I$}. Since $\M_{\min}$ and $\M_{\max}$ are subgroups of $\h$ that are stable under $\hp$, they are the smallest resp. the largest relative odd form parameter of level $I$. If $\M$ is a relative odd form parameter of level $I$, then the pair $\I=(I,\M)$ is called an {\it odd form ideal}. 
\end{definition}

\begin{remark}
In \cite{petrov}, $\M_{\max}$ was defined as $\M_{\max}=\{(u, a)\in\K \mid a\in I, B_0(u, v)\in I\text{ for any } v\in V_0\}$. Obviously the zeros should be removed. Further $V$ should be replaced by $V^{ev}$ because otherwise it is not clear that the level of an E-normal subgroup as defined in Definition \ref{deflevel} is an odd form ideal. In an e-mail to the author Petrov confirmed that the definition of $\M_{\max}$ in Definition \ref{defrelfp} above is correct. Unfortunately, Petrov's definition of $\M_{\max}$ was copied in \cite{yu-tang} and \cite{yu-li-liu}. In \cite{yu-li-liu} this led to an unnatural definition of the level of an E-normal subgroup and as a consequence to a very strange looking main result.
\end{remark}

One checks easily that any relative odd form parameter $\M$ is a normal subgroup of $\h^{ev}$. The map $Q_{\M}:V^{ev}\rightarrow \h^{ev}/\M$ induced by the map $Q$ defined in Definition \ref{defoqs} is called the {\it relative odd quadratic form on $V^{ev}$ defined by $\M$}.
\begin{definition}
Let $\I=(I,\M)$ be an odd form ideal. The subgroup
\[\U(\V, \I)=\{\sigma\in\U(\V)\mid Q_{\M}(\sigma v)=Q_{\M}(v)~\forall v\in V^{ev}\}\]
of $\U(\V)$ is called the {\it principal congruence subgroup of level $\I$}. The normalizer of $\U(\V, \I)$ in $\U(\V)$ is called the {\it normalized principal congruence subgroup of level $\I$} and is denoted by $\NU(\V, \I)$.
\end{definition}

\begin{remark}
In \cite{petrov}, the principal congruence subgroup of level $\I$ was defined as the subgroup of $\GL(V)$ consisting of all elements $g$ such that $(gv - v, (v - gv, v)_q)$ lies in $\M$ for every $v\in V$. Obviously $\GL(V)$ should be replaced by $\U(V,q)$. Further ``for every $v\in V$" should be replaced by ``for every $v\in V^{ev}$" (Petrov agreed on that). The reason for that is, that otherwise it is not clear that in the hyperbolic case the relative elementary subgroup of a certain level is contained in the principal congruence subgroup of the same level.
\end{remark}

\begin{definition}\label{defaction}
Let $\sigma\in \U(\V)$. If $\I=(I,\M)$ is an odd form ideal, set $^{\sigma}\I:=(I,{}^{\sigma}\M)$ where $^{\sigma}\M:=\{(\sigma v,x)\mid (v,x)\in \M\}$. This defines a left action of $\U(\V)$ on the set of all form ideals.
\end{definition}
\begin{proposition}\label{proptild}
Let $\sigma\in \U(\V)$ and $\I$ an be an odd form ideal. Then $^{\sigma}\!\U(\V, \I)=\U(\V, {}^{\sigma}\I)$.
\end{proposition}
\begin{proof}
Clearly it suffices to show that $^{\sigma}\!\U(\V, \I)\subseteq \U(\V, {}^{\sigma}\I)$. Let $\tau\in\U(\V, \I)$ and $v\in V^{ev}$. Then $\widetilde\sigma v\in V^{ev}$. Since $\tau$ preserves $Q_\M$, we have
\begin{align*}
(\tau\widetilde\sigma v-\widetilde\sigma v,B(\widetilde\sigma v-\tau\widetilde\sigma v,\widetilde\sigma v))\in\M.
\end{align*}
Since $\sigma$ preserves $B$, it follows that 
\begin{align*}
(\tau\widetilde\sigma v-\widetilde\sigma v,B(v-\sigma\tau\widetilde\sigma v,v))\in\M
\end{align*}
and hence
\begin{align*}
(\sigma\tau\widetilde\sigma v-v,B(v-\sigma\tau\widetilde\sigma v,v))\in{}^{\sigma}\M.
\end{align*}
Thus $\sigma\tau\widetilde\sigma $ preserves $Q_{^{\sigma}\M}$ and therefore lies in $\U(\V,{}^{\sigma}\I)$.
\end{proof}
\begin{corollary}
Let $\sigma\in \U(\V)$ and $\I$ an be an odd form ideal. Then $^{\sigma}\!\NU(\V, \I)=\NU(\V, {}^{\sigma}\I)$.
\end{corollary}

\section{The hyperbolic case}
In this section $R$ denotes a ring with pseudoinvolution $~\bar{}~$.
\subsection{Odd hyperbolic unitary groups}
\begin{definition}
Let $\V=(V, B, Q_\K)$ and $\V'=(V', B', Q_{\K'})$ denote odd quadratic spaces. Their {\it orthogonal sum}  is the odd quadratic space 
\[\V\operp \V':=(V\oplus V', B+B', Q_{\K+{\K}'})\]
where 
\[B+B'(v+v',w+w')=B(v,w)+B'(v',w')\]
and
\[\K+{\K}'=\{(v+v',x+x')\mid (v,x)\in\K, (v',x')\in {\K}'\}.\]
\end{definition}

\begin{definition}
Let $H$ be the free right $R$-module with basis $\{e_1, e_{-1}\}$. Let $B$ be the anti-Hermitian form on $H$ defined by $B_H(e_1, e_{-1})=1$ and $B_H(e_1, e_1)=0=B(e_{-1}, e_{-1})$. Further set $\K_H:=\{(e_1x_1+e_{-1}x_{-1},\bar x_1\bar 1^{-1}x_{-1}+y+\bar y)\mid x_1,x_{-1},y\in R\}$. Then $\Hy:=(H, B_H, \K_H)$ is an odd quadratic space which is called the {\it hyperbolic plane}. For any $n\in \N$ we denote by $\Hy^n$ the orthogonal sum of $n$ copies of the hyperbolic plane $\Hy$. 
\end{definition}

\begin{definition}
An {\it odd hyperbolic unitary space of rank $n$} is an orthogonal sum $\Hy^n\operp \V_0$ where $\V_0=(V_0, B_0, Q_{\K_0})$ is some odd quadratic space. The odd unitary group $\U(\Hy^n\operp \V_0)$ is called an {\it odd hyperbolic unitary group} and is denoted by $\U_{2n}(\V_0)$.
\end{definition}

\begin{definition}\label{defwit}
Let $\V=(V, B, Q_\K)$ be an odd quadratic space. A {\it hyperbolic pair} is a pair $(v,w)$ where $v$ and $w$ are isotropic elements of $V$ such that $B(v,w)=1$. The greatest $n\in \N_0$ satisfying the condition that there exist $n$ mutually orthogonal hyperbolic pairs in $V$ is called the Witt index of $\V$ and is denoted by $\ind \V$. It is easy to see that if $1\leq n\leq \ind \V$, then $\U(\V)$ is isomorphic to an odd hyperbolic unitary group $\U_{2n}(\V_0)$, cf. \cite[Section 5]{petrov}.
\end{definition}

Until the end of Section 4, $n\in \N$ denotes a nonnegative integer and $\V_0=(V_0, B_0, Q_{\K_0})$ an odd quadratic space. We set $\V:=\Hy^n\operp \V_0$ and write $\Hy^n=(V_h, B_h, Q_{\K_h})$ and $\V=(V, B, Q_\K)$ (hence $V=V_h\oplus V_0$, $B=B_h+B_0$ and $\K=\K_{h}+\K_0$). We denote the basis of $V_h$ coming from the bases of the orthogonal summands by $e_1,\dots, e_n, e_{-n}, \dots, e_{-1}$. Further we set $\Theta_+:=\{1,\dots,n\}$, $\Theta_{-}:=\{-n,\dots,-1\}$ and $\Theta:=\Theta_+\cup \Theta_-$. We define an order $<$ on $\Theta$ by $1<\dots<n<-n<\dots<-1$. We set $\epsilon_i:=\bar 1^{-1}$ if $i\in\Theta_+$ and $\epsilon_i:=-1$ if $i\in \Theta_-$.\\

If $v\in V$, then there are uniquely determined $v_h\in V_h$ and $v_0\in V_0$ such that $v=v_h+v_0$. Moreover, there are uniquely determined $v_i\in R~(i\in \Theta)$ such that $v_h=\sum\limits_{i\in\theta}e_iv_i$. Define a sesquilinear form $F_h$ on $V_h$ by 
\[F_h(v_h,w_h)=\sum\limits_{i\in\Theta_+}\bar v_i\bar 1^{-1}w_{-i}\quad\forall v_h,w_h\in V_h.\]
Then
\[B_{h}(v_h,w_h)=F_h(v_h,w_h)-\overline{F_h(w_h,v_h)}\quad\forall v_h,w_h\in V_h\]
and
\[\K_{h}=\{(v_h,F_h(v_h,v_h)+x+\bar x)\mid v_h\in V_h,x\in R\}.\]

For a $\sigma\in \U_{2n}(\V_0)$ define the linear maps
\begin{align*}
\sigma^{hh}:V_h&\rightarrow V_h,&\sigma^{h0}:V_h&\rightarrow V_0,&\sigma^{0h}:V_0&\rightarrow V_h,&\sigma^{00}:V_0&\rightarrow V_0.\\
v_h&\mapsto (\sigma v_h)_h &v_h&\mapsto(\sigma v_h)_0&v_0&\mapsto (\sigma v_0)_h &v_0&\mapsto (\sigma v_0)_0
\end{align*} 
Since $V_h$ is a free right $R$-module of rank $2n$, we can identify $\sigma^{hh}$ with a matrix in $\Mat_{2n}(R)$. Instead of $\sigma_{ij}^{hh}$ we write just $\sigma_{ij}$ for the entry of $\sigma^{hh}$ at position $(i,j)$ (which equals $(\sigma e_j)_i$).\\

\begin{lemma}\label{lemodd}
Let $\sigma \in \GL(V)$. Then $\sigma\in U_{2n}(\V_0)$ iff Conditions (i)-(vi) below are satisfied.
\begin{enumerate}[(i)]
\item $\widetilde\sigma_{ij}=-\epsilon_i\bar\sigma_{-j,-i}\epsilon_{-j}$ for any $i,j\in\Theta$.
\item $(\widetilde\sigma^{0h}v_0)_{i}=-\epsilon_iB_0(\sigma^{h0}e_{-i},v_0)$ for any $v_0\in V_0$ and $i\in\Theta$.
\item $B_0(v_0,\widetilde\sigma^{h0}e_i)=\overline{(\sigma v_0)_{-i}}\epsilon_{-i}$ for any $v_0\in V_0$ and $i\in\Theta$.
\item $B_0(v_0,\widetilde\sigma^{00}w_0)=B_0(\sigma^{00}v_0,w_0)$ for any $v_0,w_0\in V_0$.
\item $Q(\sigma e_i)\in \K$ for any $i\in\Theta$.
\item $Q(\sigma v_0)\minus Q(v_0)\in \K$ for any $v_0\in V_0$.
\end{enumerate}
\end{lemma}
\begin{proof}
($\Rightarrow$) Suppose that $\sigma \in U_{2n}(\V_0)$. Then 
\begin{align*}
\widetilde\sigma_{ij}=-\epsilon_iB(e_{-i},\widetilde\sigma e_j)=-\epsilon_iB(\sigma e_{-i},e_j)=-\epsilon_i\bar\sigma_{-j,-i}\epsilon_{-j}
\end{align*}
for any $i,j\in \Theta$ since $\sigma$ preserves $B$. Hence (i) holds. Conditions (ii)-(iv) can be shown similarly. Conditions (v) and (vi) are satisfied since $\sigma$ preserves $Q_\K$ and $Q(e_i)\in \K_h\subseteq \K$ for any $i\in \Theta$.\\
($\Leftarrow$) Suppose $\sigma$ satisfies (i)-(vi). It is an easy exercise to deduce from Conditions (i)-(iv) that $B(\sigma v,\sigma w)=B(v,w)$ for any $v,w\in V$ (hint: replace $w$ by $\widetilde\sigma w'$). It remains to show that $Q(\sigma v)\minus Q(v)\in \K$ for any $v\in V$. A straightforward computation shows that
\begin{align}
&Q(\sigma v)\minus Q(v)\notag\\
=&(\plus\limits_{i=1}^{-1}(Q(\sigma e_i)\minus Q(e_i))\hp v_i)\+Q(\sigma v_0)\minus Q(v_0)\notag\\
&\+(0,\sum\limits_{i<j}\bar v_i\epsilon_i(\sigma_{-i,j}-\delta_{-i,j})v_{-i})+\overline{\sum\limits_{i<j}\bar v_i\epsilon_i(\sigma_{-i,j}-\delta_{-i,j})v_{-i}})\notag\\
&\+(0,\sum\limits_{i\in\Theta}\bar v_i\epsilon_i(\sigma^{0h} v_0)_{-i}+\overline{\sum\limits_{i\in\Theta}\bar v_i\epsilon_i(\sigma^{0h} v_0)_{-i}}).
\end{align}
It follows from (v) and (vi) that $Q(\sigma v)\minus Q(v)\in \K$.
\end{proof}
\begin{definition}
Let $\h_0$ be the Heisenberg group of the anti-Hermitian space $(V_0, B_0)$. Define the map 
\begin{align*}
Q^0:V&\rightarrow \h_0\\
v&\mapsto (v_0,-F_h(v_h, v_h))
\end{align*}
One checks easily that.
\begin{align*}
&Q^0(vx)=Q^0(v)\hp x~\forall v\in V, x\in R\text{ and}
\\
&Q^0(v+w)=Q^0(v)\+Q^0(w)\minus(0,B(v,w)+F_h(w_h,v_h)+\overline{F_h(w_h,v_h)})~\forall v,w\in V.
\end{align*}
\end{definition}
\begin{lemma}\label{lemtec}
Let $v,w\in V$. Then $Q(v)\equiv Q(w)\bmod \K~\Leftrightarrow~ Q^0(v)\equiv Q^0(w)\bmod \K_0$.
\end{lemma}
\begin{proof}
Clearly
\begin{align*}
&Q(v)\minus Q(w)\in \K\\
\Leftrightarrow~&(v-w,B(w-v,w))\in \K\\
\Leftrightarrow~&(v-w,B(w-v,w))\minus (v_h-w_h,F_H(v_h-w_h,v_h-w_h))\in \K\\
\Leftrightarrow~& Q^0(v-w)\+(0,B(w-v,w))\in \K\\
\Leftrightarrow~& Q^0(v-w)\+(0,B(w-v,w))\in \K_0
\end{align*}
(note that $(v_h-w_h,F_H(v_h-w_h,v_h-w_h))\in \K_h\subseteq \K$). One checks easily that 
\begin{align*}
&Q^0(v-w)\+(0,B(w-v,w))\minus (Q^0(v)\minus Q^0(w))\\
=&(0,-F_h(w_h,w_h-v_h)-\overline{F_h(w_h,w_h-v_h)})\in(\K_0)_{\min}\subseteq\K_0.
\end{align*}
The assertion of the lemma follows.
\end{proof}

\begin{proposition}\label{propodd}
Let $\sigma \in \GL(V)$. Then $\sigma\in U_{2n}(\V_0)$ iff Conditions (i)-(vi) below are satisfied.
\begin{enumerate}[(i)]
\item $\widetilde\sigma_{ij}=-\epsilon_i\bar\sigma_{-j,-i}\epsilon_{-j}$ for any $i,j\in\Theta$.
\item $(\widetilde\sigma^{0h}v_0)_{i}=-\epsilon_iB_0(\sigma^{h0}e_{-i},v_0)$ for any $v_0\in V_0$ and $i\in\Theta$.
\item $B_0(v_0,\widetilde\sigma^{h0}e_i)=\overline{(\sigma v_0)_{-i}}\epsilon_{-i}$ for any $v_0\in V_0$ and $i\in\Theta$.
\item $B_0(v_0,\widetilde\sigma^{00}w_0)=B_0(\sigma^{00}v_0,w_0)$ for any $v_0,w_0\in V_0$.
\item $Q^0(\sigma e_i)\in \K_0$ for any $i\in\Theta$.
\item $Q^0(\sigma v_0)\minus Q^0(v_0)\in \K_0$ for any $v_0\in V_0$.
\end{enumerate}
\end{proposition}
\begin{proof}
Follows from the Lemmas \ref{lemodd} and \ref{lemtec}.
\end{proof}

\subsection{The elementary subgroup and relative elementary subgroups}
\begin{definition}\label{defel}
If $i,j\in \Theta$, $i\neq\pm j$ and $x\in R$, then the ESD transvection \[T_{ij}(x):=T_{e_{-j},-e_ix\epsilon_j}(0)\]
is called an {\it elementary short root transvection}. If $i\in \Theta$ and $(v_0,x)\in \K_0$, then the ESD transvection \[T_{i}(v_0,x):=T_{e_{i},v\epsilon_{-i}}(\bar \epsilon_{-i}\bar 1^{-1}x\epsilon_{-i})\]
is called an {\it elementary extra short root transvection}.
If an element of $\U_{2n}(\V_0)$ is an elementary short or extra short root transvection, then it is called an {\it elementary transvection}. The subgroup of $\U_{2n}(\V_0)$ generated by all elementary transvections is called the {\it elementary subgroup} of $\U_{2n}(\V_0)$ and is denoted by $\EU_{2n}(\V_0)$. 
\end{definition}

One checks easily that an elementary short root transvection $T_{ij}(x)$ maps
\begin{align*}
e_k\mapsto e_k~(k\neq -i,j), \quad e_{-i}\mapsto e_{-i}+e_{-j}\epsilon_{-j}\bar x\epsilon_i, \quad e_j\mapsto e_j+e_ix, \quad w_0\mapsto w_0~(w_0\in V_0).
\end{align*}
Hence 
\begin{align*}
T_{ij}^{hh}(x)=e+e^{ij}x+e^{-j,-i}\epsilon_{-j}\bar x\epsilon_i,\quad T_{ij}^{h0}(x)=0=T_{ij}^{0h}(x),\quad T_{ij}^{00}(x)=id_{V_0}
\end{align*}
where $e$ denotes the identity matrix in $M_{2n}(R)$ and $e^{pq}$ the matrix in $M_{2n}(R)$ with a $1$ at position $(p,q)$ and zeros elsewhere.\\

One checks easily that an elementary extra short root transvection $T_{i}(v_0,x)$ maps
\begin{align*}
e_k\mapsto e_k~(k\neq -i), \quad e_{-i}\mapsto e_{-i}+e_i\epsilon_{i} x-v_0,\quad w_0\mapsto w_0-e_i\epsilon_iB_0(v_0,w_0)~(w_0\in V_0).
\end{align*}
Hence 
\begin{align*}
&T_{i}^{hh}(v_0,x)=e+e^{i,-i}\epsilon_ix,\quad T_{i}^{h0}(v_0,x) w_h=-v_0w_{-i}~(w_h\in V_h),\\
&T_{i}^{0h}(v_0,x)w_0=-e_i\epsilon_i B_0(v_0,w_0)~(w_0\in V_0), \quad T_{i}^{00}(v_0,x)=id_{V_0}.
\end{align*}
\begin{lemma}\label{39}
The following relations hold for the elementary transvections.
\begin{align*}
&T_{ij}(x)=T_{-j,-i}(\epsilon_{-j}\bar x\epsilon_{i}), \tag{R0}\\
&T_{ij}(x)T_{ij}(y)=T_{ij}(x+y), \tag{R1}\\
&T_{i}(v_0,x)T_{i}(w_0,y)=T_{i}((v_0,x)\+(w_0,y)), \tag{R2}\\
&[T_{ij}(x),T_{hk}(y)]=id_V \text{ if } h\neq j,-i \text{ and } k\neq i,-j, \tag{R3}\\
&[T_{i}(v_0,x),T_{jk}(y)]=id_V \text{ if } i\neq -j,k,\tag{R4}\\
&[T_{ij}(x),T_{jk}(y)]=T_{ik}(xy) \text{ if } i\neq\pm k, \tag{R5}\\
&[T_{i}(v_0,x),T_{j}(w_0,y)]=T_{i,-j}(\epsilon_iB_0(v_0,w_0)) \text{ if } i\neq\pm j, \tag{R6}\\
&[T_{i}(v_0,x),T_{i}(w_0,y)]=T_{i}(0,B_0(v_0,w_0)+\overline{B_0(v_0,w_0)}), \tag{R7}\\
&[T_{i}(v_0,x),T_{-i,j}(y)]=T_{ij}(\epsilon_ixy)T_{-j}((v_0,-\bar x)\hp y)\text{ and}\tag{R8}\\
&[T_{ij}(x),T_{j,-i}(y)]=T_{i}(0,-\epsilon_{-i}\bar 1xy-\overline{\epsilon_{-i}\bar 1xy}). \tag{R9}
\end{align*}
\end{lemma}
\begin{proof}
Straightforward computation.
\end{proof}
\begin{definition}\label{42}
If $i,j\in\Theta$ such that $i\neq\pm j$, then the element 
\begin{align*}
P_{ij}:=T_{ij}(1)T_{ji}(-1)T_{ij}(1)                                                                                                                                                                                                                                                                                                                                                                                                                                                                                             
\end{align*}
of $\EU_{2n}(\V_0)$ is called an {\it elementary permutation}.
One checks easily that 
\begin{align*}
&P^{hh}_{ij}=e-e^{ii}-e^{jj}-e^{-i,-i}-e^{-j,-j}+e^{ij}-e^{ji}+e^{-i,-j}\epsilon_{-i}\epsilon_{-j}^{-1}-e^{-j,-i}\epsilon_{-j}\epsilon_{-i}^{-1},\\
&P^{h0}_{ij}=0=P^{0h}_{ij},\quad P_{ij}^{00}=id_{V_0}.
\end{align*}
It is easy to show that $\widetilde P_{ij}=P_{ji}$. 
\end{definition}
\begin{lemma}\label{43}
Let $i,j,k\in\Theta$ such that $i\neq \pm j$ and $k\neq \pm i,\pm j$. Let $x\in R$ and $(v_0,x)\in \K_0$. Then 
\begin{enumerate}[(i)]
\item $^{P_{ki}}T_{ij}(x)=T_{kj}(x)$,
\item $^{P_{kj}}T_{ij}(x)=T_{ik}(x)$ and
\item $^{P_{-k,-i}}T_{i}(v_0,x)=T_{k}(v_0,x)$.
\end{enumerate}
\end{lemma}
\begin{proof}
Follows from the relations in Lemma \ref{39}.
\end{proof}

\begin{definition}
Let $\I_0=(I,\M_0)$ denote an odd form ideal on $\V_0$. An elementary short root transvection $T_{ij}(x)$ is called {\it $\I_0$-elementary} if $x\in I$. An elementary extra short root transvection $T_{i}(\xi)$ is called {\it $\I_0$-elementary} if $\xi\in \M_0$. If an element of $U_{2n}(\V_0)$ is an $\I_0$-elementary short or extra short root transvection, then it is called an {\it $\I_0$-elementary transvection}. The subgroup $\EU_{2n}(\I_0)$ of $\EU_{2n}(\V_0)$ generated by the $\I_0$-elementary transvections is called the {\it preelementary subgroup of level $\I_0$}. Its normal closure $\EU_{2n}(\V_0,\I_0)$ in $\EU_{2n}(\V_0)$ is called the {\it elementary subgroup of level $\I_0$}.
\end{definition}

\subsection{Congruence subgroups}
If $\I_0=(I,\M_0)$ is an odd form ideal on $\V_0$, then $\I:=(I,\M)$ where $\M=\K_h\hp I\+\M_0$ is an odd form ideal on $\V$. Conversely, if $\I=(I,\M)$ is an odd form ideal on $\V$, then $\I_0:=(I,\M_0)$ where $\M_0=\{Q^0(v)\+(0,x)\mid (v,x)\in M\}$ is an odd form ideal on $\V_0$. This yields a 1-1 corespondence between odd form ideals on $\V_0$ and odd form ideals on $\V$.\\

In this subsection $\I_0=(I,\M_0)$ denotes an form ideal on $\V_0$ and $\I=(I,\M)$ the corresponding odd form ideal on $\V$. We denote the principal congruence subgroup $\U(\V,\I)$ by $\U_{2n}(\V_0, \I_0)$ and the normalized principal congruence subgroup $\NU(\V,\I)$ by $\NU_{2n}(\V_0, \I_0)$. The even part of $V_0$ is denoted by $V^{ev}_0$.
\begin{lemma}\label{lemoddrel}
Let $\sigma \in \U_{2n}(\V_0)$. Then $\sigma\in \U_{2n}(\V_0,\I_0)$ iff Conditions (i) and (ii) below are satisfied.
\begin{enumerate}[(i)]
\item $Q(\sigma e_i)\minus Q(e_i)\in \M$ for any $i\in\Theta$.
\item $Q(\sigma v_0)\minus Q(v_0)\in \M$ for any $v_0\in V^{ev}_0$.
\end{enumerate}
\end{lemma}
\begin{proof}
($\Rightarrow$) Suppose that $\sigma \in U_{2n}(\V_0,\I_0)$. Then clearly $Q(\sigma e_i)\minus Q(e_i)\in \M$ for any $i\in\Theta$ and $Q(\sigma v_0)\minus Q(v_0)\in \M$ for any $v_0\in V^{ev}_0$ since $\sigma$ preserves $Q_\M$. \\
($\Leftarrow$) Suppose that Conditions (i) and (ii) are satisfied. If $i\in \Theta$, then $Q(\sigma e_i)\minus Q(e_i)=(\sigma e_i-e_i,x)$ for some $x\in R$. Since $Q(\sigma e_i)\minus Q(e_i)\in \M\subseteq\M_{\max}$, it follows that $B(e_j,\sigma e_i-e_i)\in I$ for any $j\in \Theta$. But one checks easily that $B(e_j,\sigma e_i-e_i)=-\epsilon_{-j}^{-1}(\sigma_{-j,i}-\delta_{-j,i})$. Thus $\sigma^{hh}\equiv e\bmod I$. If $v_0\in V^{ev}_0$, then $Q(\sigma v_0)\minus Q(v_0)=(\sigma v_0-v_0,x)$ for some $x\in R$. Since $Q(\sigma v_0)\minus Q(v_0)\in \M\subseteq\M_{\max}$, it follows that $B(e_j,\sigma v_0-v_0)\in I$ for any $j\in \Theta$. But one checks easily that $B(e_j,\sigma v_0-v_0)=-\epsilon_{-j}^{-1}(\sigma^{0h}v_0)_{-j}$ and therefore $(\sigma^{0h}v_0)_{-j}\in I$. Now it follows from Equation (5) in the proof of Lemma \ref{lemodd} that $Q(\sigma v)\minus Q(v)\in \M$ for any $v\in V^{ev}$. Thus $\sigma\in U_{2n}(\V_0,\I_0)$.
\end{proof}

\begin{lemma}\label{lemtecrel}
Let $v,w\in V$ such that $v_h-w_h\in V_hI$ (i.e. $v_i-w_i\in I$ for any $i\in \Theta$). Then 
\[Q(v)\equiv Q(w)\bmod \M~\Leftrightarrow~Q^0(v)\equiv Q^0(w)\bmod \M_0.\]
\end{lemma}
\begin{proof}
See the proof of Lemma \ref{lemtec}.
\end{proof}

\begin{proposition}\label{propoddrel}
Let $\sigma \in \U_{2n}(\V_0)$. Then $\sigma\in \U_{2n}(\V_0,\I_0)$ iff the Conditions (i)-(iii) below are satisfied.
\begin{enumerate}[(i)]
\item $\sigma^{hh}\equiv e\bmod I$.
\item $Q^0(\sigma e_i)\in \M_0$ for any $i\in\Theta$.
\item $Q^0(\sigma v_0)\minus Q^0(v_0)\in \M_0$ for any $v_0\in V^{ev}_0$.
\end{enumerate}
\end{proposition}
\begin{proof}
Follows from the Lemmas \ref{lemoddrel} and \ref{lemtecrel}.
\end{proof}

\begin{corollary}\label{cor1}
$\EU_{2n}(\I_0)\subseteq \U_{2n}(\V_0,\I_0)$
\end{corollary}
\begin{proof}
One checks easily that any $\I_0$-elementary transvection satisfies Conditions (i)-(iii) in Proposition \ref{propoddrel}.
\end{proof}

\begin{corollary}\label{corobv}
Let $\mathfrak{N}_0\subseteq \K_0$ be another relative odd form parameter of level $I$. Then 
\[\U_{2n}(\V_0,(I,\M_0))=\U_{2n}(\V_0,(I,\mathfrak{N}_0))~\Leftrightarrow~\M_0=\mathfrak{N}_0.\]
\end{corollary}
\begin{proof}
Suppose that $\U_{2n}(\V_0,(I,\M_0))=\U_{2n}(\V_0,(I,\mathfrak{N}_0))$. Let $(v_0,x)\in \M_0$. Then $T_{-1}(v_0,x)\in \U_{2n}(\V_0,$ $(I,\M_0))=\U_{2n}(\V_0,(I,\mathfrak{N}_0))$ by the previous corollary. It follows from Proposition \ref{propoddrel} that $(-v_0,x)=Q^0(T_{-1}(v_0,x)e_1)\in \mathfrak{N}_0$. Hence $(v_0,x)=(-v_0,x)\hp(-1)\in\mathfrak{N}_0$ and therefore we have shown $\M_0\subseteq \mathfrak{N}_0$. Analogously one can show that $\mathfrak{N}_0\subseteq \M_0$.
\end{proof}

By Definition \ref{defaction} we have a left action of $\U_{2n}(\V_0)$ on the set of all form ideals on $\V$. This action induces a left action of $\U_{2n}(\V_0)$ on the set of all form ideals on $\V_0$, given by $(\sigma,\I_0)\mapsto {}^{\sigma}\I_0:=(I,{}^{\sigma}\M_0)$ where \[{}^{\sigma}\M_0=\{Q^0(\sigma v_0)\+(0,x)\mid (v_0,x)\in \M_0\}.\]
\begin{proposition}\label{proptildhyp}
$\NU_{2n}(\V_0,\I_0)=\{\sigma\in \U_{2n}(\V_0)\mid {}^{\sigma}\I_0=\I_0\}$.
\end{proposition}
\begin{proof}
Follows from Proposition \ref{proptild} and Corollary \ref{corobv}.
\end{proof}

\begin{remark}
On page 4760 in \cite{petrov}, Petrov defined the set $\tilde \U(V,q):=\{g\in \U(V,q)\mid (g w,a)\in \I~\forall (w,a)\in\I \}$ and claimed that it is a group. But is is not clear that $\tilde \U(V,q)$ is closed under taking inverses. In an e-mail to the author, Petrov wrote that he ``probably missed the condition that every $g$ from $\tilde \U(V,q)$ must preserve $U$ (that is $gU=U$)" where $U$ is the submodule of $V$ generated by $VI$ and all elements $v\in V$ such that $(v, x)\in \M$ for some $x\in I$. After making this correction, $\tilde \U(V,q)$ conincides with the normalized principal subgroup $\NU_{2n}(\V_0,\I_0)$ (follows from Proposition \ref{proptildhyp} above).
\end{remark}

\begin{corollary}\label{cor2}
$\EU_{2n}(\V_0)\subseteq \NU_{2n}(\V_0,\I_0)$.
\end{corollary}
\begin{proof}
Let $\sigma$ be an elementary transvection. Then $Q^0(\sigma v_0)=Q^0(v_0)$ for any $v_0\in V_0$ (see the two paragraphs after Definition \ref{defel}). Hence ${}^{\sigma}\I_0=\I_0$.
\end{proof}

\begin{corollary}
$\EU_{2n}(\V_0,\I_0)\subseteq \U_{2n}(\V_0,\I_0)$.
\end{corollary}
\begin{proof}
Follows from Corollaries \ref{cor1} and \ref{cor2}.
\end{proof}

\begin{definition}
The subgroup 
\[\CU_{2n}(\V_0,\I_0):=\{\sigma\in \NU_{2n}(\V_0,\I_0)\mid [\sigma,\EU_{2n}(\V_0)]\subseteq \U_{2n}(\V_0,\I_0)\}\]
of $\U_{2n}(V_0)$ is called the {\it full congruence subgroup of level $\I_0$}.
\end{definition}

\subsection{The standard commutator formula}
In this subsection we assume that $R$ is commutative and $n\geq 3$. Further $\I_0$ denotes an odd form ideal on $\V_0$.
\begin{theorem}[V. A. Petrov, 2005]\label{thmscf}
$\EU_{2n}(\V_0,\I_0)$ is a normal subgroup of $\NU_{2n}(\V_0,\I_0)$ and the standard commutator formula
\begin{align*}
[\CU_{2n}(\V_0,\I_0),\EU_{2n}(\V_0)]
=[\EU_{2n}(\V_0,\I_0),\EU_{2n}(\V_0)]
=\EU_{2n}(\V_0,\I_0)
\end{align*}
holds. In particular from the absolute case $\I_0=(R,\K_0)$ it follows that $\EU_{2n}(\V_0)$ is perfect and normal in $\U_{2n}(\V_0)$
\end{theorem}
\begin{proof}
Follows from \cite[Proposition 4 and Corollary on page 4765]{petrov} and \cite[Lemma 1]{stepanov} (note that in \cite{petrov} the full congruence subgroup is defined a little differently).
\end{proof}

\begin{corollary}\label{165}
If $\sigma\in \U_{2n}(\V_0)$, then 
\[^{\sigma}\EU_{2n}(\V_0,\I_0)=\EU_{2n}(\V_0,{}^{\sigma}\I_0)\]
and
\[^{\sigma}\CU_{2n}(\V_0,\I_0)=\CU_{2n}(\V_0,{}^{\sigma}\I_0).\]
Further $\NU_{2n}(\V_0,\I_0)$ is the normalizer of $\EU_{2n}(\V_0,\I_0)$ and of $\CU_{2n}(\V_0,\I_0)$.
\end{corollary}
\begin{proof}
Follows from Proposition \ref{proptild} and the standard commutator formula in Theorem \ref{thmscf}. 
\end{proof}

\begin{definition}\label{deflevel}
A subgroup $H$ of $\U_{2n}(\V_0)$ is called {\it E-normal} iff it is normalized by the elementary subgroup $\EU_{2n}(\V_0)$. Suppose $H$ is an E-normal subgroup of $\U_{2n}(\V_0)$. Set
\[I:=\{x\in R\mid T_{ij}(x)\in H \text{ for some }i,j\in\Theta\}\]
and
\[\M_0:=\{(v_0,y)\in \K_0\mid T_{i}(v_0,y)\in H \text{ for some }i\in\Theta\}.\]
Then $\I_0:=(I,\M_0)$ is an odd form ideal on $\V_0$ such that $\EU_{2n}(\V_0,\I_0)\subseteq H$. It is called the {\it level of $H$} and $H$ is called an {\it E-normal subgroup of level $\I_0$}. 
\end{definition}

\begin{corollary}\label{cornu}
Let $H$ be an E-normal subgroup of level $\I_0$. Then for any $\sigma\in \U_{2n}(\V_0)$, $^{\sigma}H$ is an E-normal subgroup of level ${}^{\sigma}\I_0$. It follows that $H\subseteq \NU_{2n}(\V_0,\I_0)$.
\end{corollary}
\begin{proof}
Follows from the previous corollary and Proposition \ref{proptildhyp}.
\end{proof}

\subsection{The case $\M_0=(\M_0)_{\max}$}
In this subsection we assume that $R$ is commutative and $n\geq 3$. Further $(\I_0)_{\max}=(I,(\M_0)_{\max})$ denotes an odd form ideal on $\V_0$ with maximal relative odd form parameter. Proposition \ref{proptildhyp} implies that $\NU_{2n}(\V_0,(\I_0)_{\max})=\U_{2n}(\V_0)$, i.e. $\U_{2n}(\V_0,(\I_0)_{\max})$ is normal in $\U_{2n}(\V_0)$.
\begin{lemma}\label{lemumax}
Let $\sigma \in \U_{2n}(\V_0)$. Then $\sigma\in \U_{2n}(\V_0,(\I_0)_{\max})$ iff Conditions (i)-(iv) below are satisfied.
\begin{enumerate}[(i)]
\item $\sigma^{hh}\equiv e\bmod I$.
\item $(\sigma^{0h}v_0)_i \in I$ for any $v_0\in V^{ev}_0$ and $i\in\Theta$.
\item $B_0(\sigma^{h0}e_i,v_0)\in I$ for any $v_0\in V^{ev}_0$ and $i\in\Theta$.
\item $B_0((\sigma^{00}-id_{V_0})v_0,w_0)\in I$ for any $v_0,w_0\in V^{ev}_0$.
\end{enumerate}
\end{lemma}
\begin{proof}
Follows from Proposition \ref{propoddrel}.
\end{proof}

We leave it to the reader to deduce the next lemma from the last one.
\begin{lemma}\label{lemgen}
Let $\sigma, \tau\in \U_{2n}(\V_0)$. Then $\sigma\equiv \tau \bmod \U_{2n}(\V_0,(\I_0)_{\max})$ iff Conditions (i)-(iv) below are satisfied.
\begin{enumerate}[(i)]
\item $\sigma^{hh}\equiv \tau^{hh}\bmod I$.
\item $((\sigma^{0h}-\tau^{0h})v_0)_i \in I$ for any $v_0\in V^{ev}_0$ and $i\in\Theta$.
\item $B_0((\sigma^{h0}-\tau^{h0})e_i,v_0)\in I$ for any $v_0\in V^{ev}_0$ and $i\in\Theta$.
\item $B_0((\sigma^{00}-\tau^{00})v_0,w_0)\in I$ for any $v_0,w_0\in V^{ev}_0$.
\end{enumerate}
\end{lemma}

\begin{proposition}\label{propcumax}
Let $\sigma\in\U_{2n}(\V_0)$. Then $\sigma\in \CU_{2n}(\V_0,(\I_0)_{\max})$ iff Conditions (i)-(v) below are satisfied.
\begin{enumerate}[(i)]
\item $\sigma_{ij}\in I$ for any $i,j\in \Theta$ such that $i\neq j$.
\item $\sigma_{ii}-\sigma_{jj}\in I$ for any $i,j\in \Theta$ such that $i\neq j$.
\item $(\sigma^{0h}v_0)_i \in I$ for any $v_0\in V^{ev}_0$ and $i\in\Theta$.
\item $B_0(\sigma^{h0}e_i,v_0)\in I$ for any $v_0\in V^{ev}_0$ and $i\in\Theta$.
\item $B_0((\sigma^{00}-id_{V_0}\sigma_{ii})v_0,w_0)\in I$ for any $v_0,w_0\in V^{ev}_0$ and $i\in \Theta$.
\end{enumerate}
\end{proposition}
\begin{proof}
($\Rightarrow$) Suppose that $\sigma\in \CU_{2n}(\V_0,(\I_0)_{\max})$. Let $\tau\in \EU_{2n}(\V_0)$. Then $[\sigma,\tau]\in  \U_{2n}(\V_0,(\I_0)_{\max})$ by the definition of $\CU_{2n}(\V_0,(\I_0)_{\max})$. Hence $\sigma\tau\equiv \tau\sigma \bmod \U_{2n}(\V_0,(\I_0)_{\max})$. By letting $\tau$ vary over all elementary transvections $T_{ij}(1)~(i,j\in\Theta,i\neq\pm j)$ and $T_{i}(v_0,y)~(i\in\Theta,(v_0,y)\in\K_0)$ and applying Lemma \ref{lemgen} one gets Conditions (i)-(v).\\
($\Leftarrow$) Suppose that Conditions (i)-(v) are satisfied. Then it follows from Lemma \ref{lemgen}  that $[\sigma,\tau]\in  \U_{2n}(\V_0,$ $(\I_0)_{\max})$ for any elementary transvection $\tau$. Since $\U_{2n}(\V_0,(\I_0)_{\max})$ is normal, it follows that $\sigma \in \CU_{2n}(\V_0,$ $(\I_0)_{\max})$.
\end{proof}

The following lemma will be used in the proof of the Sandwich Classification Theorem \ref{thmsct}.
\begin{lemma}\label{lemcu}
Let $\I_0=(I,\M_0)$ be an odd form ideal. Let $\sigma\in \NU_{2n}(\V_0,\I_0)\cap \CU_{2n}(\V_0, (\I_0)_{\max})$ such that $Q^0(\sigma e_i)\in\M_0$ for any $i\in\Theta$ and further $Q^0([\sigma,\tau] e_i)\in\M_0$ for any elementary transvection $\tau$ and $i\in\Theta$. Then $\sigma \in\CU_{2n}(\V_0,\I_0)$.
\end{lemma}
\begin{proof}
Let $\tau$ be an elementary transvection and set $\xi:=[\sigma,\tau]$. In order to prove that $\sigma \in\CU_{2n}(\V_0,\I_0)$ it suffices to show that $\xi\in \U_{2n}(\V_0,\I_0)$ (note that $\U_{2n}(\V_0,\I_0)$ is E-normal by Corollary \ref{cor2}). In view of Proposition \ref{propoddrel} it suffices to show that $Q^0(\xi v_0)\minus Q^0(v_0)\in \M_0$ for any $v_0\in V^{ev}_0$.\\
\\
\underline{Case 1} Suppose that $\tau=T_{ij}(x)$. A straightforward computation shows that 
\[Q^0(\xi v_0)\minus Q^0(v_0)=Q^0(\sigma e_i)\hp a\+Q^0(\sigma e_{-j})\hp b\+(0,c\+\bar c)\]
for some $a,b,c\in I$. Hence $Q^0(\xi v_0)\minus Q^0(v_0)\in\M_0$.\\
\\
\underline{Case 2} Suppose that $\tau=T_{i}(w_0,x)$. A straightforward computation shows that 
\[Q^0(\xi v_0)\minus Q^0(v_0)=Q^0(\sigma e_i)\hp a\+(Q^0(\sigma w_0)\minus Q^0(w_0))\hp b\+(w_0,x)\hp b\+(0,c\+\bar c)\]
for some $a\in R$ and $b,c\in I$. Hence $Q^0(\xi v_0)\minus Q^0(v_0)\in\M_0$.\\
\end{proof}
\section{Sandwich classification of E-normal subgroups}
In this section $R$ denotes a commutative ring with pseudoinvolution $~\bar{}~$. Let $\V$ denote an odd quadratic space over $(R,~\bar{}~)$ such that $\ind \V\geq 3$. As mentioned in Definition \ref{defwit}, $\U(\V)$ is isomorphic to an odd hyperbolic unitary group $\U_{6}(\V_0)$ where $\V_0$ is some odd quadratic space. The main goal of this section is to classify the E-normal subgroups of $\U_{6}(\V_0)$. We start with some easy-to-check lemmas which will be used in the proof of Theorem \ref{thmm}.

\begin{lemma}\label{pre}
Let $G$ be a group and $a,b,c\in G$. Then ${}^{b^{-1}}[a,bc]=[b^{-1},a][a,c]$.
\end{lemma}

The following lemma will be used in the proof of Theorem \ref{thmm} without explicit reference.
\begin{lemma}\label{pre2}
Let $G$ be a group, $E$ a subgroup and $a\in G$. Suppose that $b\in G$ is a product of $n$ elements of the form $^{\epsilon}a^{\pm 1}$ where $\epsilon\in E$. Then for any $\epsilon'\in E$
\begin{enumerate}[(i)]
\item $^{\epsilon'}b$ is a product of $n$ elements of the form $^{\epsilon}a^{\pm 1}$ and
\item $[\epsilon',b]$ is a product of $2n$ elements of the form $^{\epsilon}a^{\pm 1}$.
\end{enumerate}
\end{lemma}

\begin{lemma}\label{lemesdspec}
Let $v\in V$ be an isotropic element such that $v_{-1}=0$. Then 
\[T_{e_1v}(0)=T_1(Q^0(v)\hp(-1)\+(0,\bar 1 v_1+\overline{\bar 1 v_1}))\prod\limits_{i=2}^{-2}T_{i,-1}(v_i)\in \EU_{6}(\V_0).\]
\end{lemma}

\begin{lemma}\label{last}
Let $(v_0,y)\in \K_0$ and $x_1,\dots,x_n\in R$. Then
\[(v_0,y)\hp \sum\limits_{i=1}^nx_i=(\plus\limits_{i=1}^n (v_0,y)\hp x_i)\+(0,\sum\limits_{i>j} \bar x_i\bar 1^{-1}yx_j+\overline{\bar x_i\bar 1^{-1}yx_j}).\]
\end{lemma}

\begin{lemma}\label{new}
Let $\sigma\in U_{2n}(\V_0)$ and $i,j\in \Theta$ such that $i\neq \pm j$. Then
\[Q^0(^{P_{ij}}\!\sigma e_i)=
\begin{cases}
Q^0(\sigma e_j), \text{ if } i,j\in \Theta_+ \text{ or }i,j\in \Theta_-,\\
Q^0(\sigma e_j)\+(0,-(\bar{\sigma}_{ij}\sigma_{-i,j}+\bar{\sigma}_{-j,j}\sigma_{jj}+\overline{\bar{\sigma}_{ij}\sigma_{-i,j}+\bar{\sigma}_{-j,j}\sigma_{jj}}~)), \text{ if } i\in \Theta_+,j\in\Theta_-,\\
Q^0(\sigma e_j)\+(0,-(\bar{\sigma}_{-i,j}\sigma_{ij}+\bar{\sigma}_{jj}\sigma_{-j,j}+\overline{\bar{\sigma}_{-i,j}\sigma_{ij}+\bar{\sigma}_{jj}\sigma_{-j,j}}~)), \text{ if } i\in \Theta_-,j\in\Theta_+.
\end{cases}
\]
\end{lemma}

\begin{definition}
Let $\sigma\in \U_{6}(\V_0)$. Then an element $\zeta\in \U_{6}(\V_0)$ is called an {\it elementary $\sigma$-conjugate} if $\zeta={}^{\tau}\sigma$ or $\zeta={}^{\tau}\widetilde\sigma$ for some $\tau\in \EU_{6}(\V_0)$.
\end{definition}
\begin{theorem}\label{thmm}
Let $\sigma\in \U_{6}(\V_0)$ and $i,j,k,l\in\Theta$ such that $k\neq\pm l$ and $i\neq \pm j$. Further let $v_0\in V_0^{ev}$. Then 
\begin{enumerate}[(i)]
\item $T_{kl}(\sigma_{ij})$ is a product of $160$ elementary $\sigma$-conjugates,
\item $T_{kl}(\sigma_{i,-i})$ is a product of $320$ elementary $\sigma$-conjugates,
\item $T_{kl}((\sigma^{0h}v_0)_i)$ is a product of $480$ elementary $\sigma$-conjugates,
\item $T_{kl}(B_0(v_0,\sigma^{h0}e_j))$ is a product of $480$ elementary $\sigma$-conjugates,
\item $T_{kl}(\sigma_{ii}-\sigma_{jj})$ is a product of $480$ elementary $\sigma$-conjugates,
\item $T_{kl}(\sigma_{ii}-\sigma_{-i,-i})$ is a product of $960$ elementary $\sigma$-conjugates,
\item $T_{k}(Q^0(\sigma e_j))$ is a product of $10564$ elementary $\sigma$-conjugates and
\item $T_{k}((\sigma_{00}-id_{V_0}\sigma_{jj})v_0,x)$ is a product of $34092$ elementary $\sigma$-conjugates for some $x\in R$.
\end{enumerate}
\end{theorem}
\begin{proof}(i) Let $x\in R$. In Step 1 we show that $T_{kl}(x\bar\sigma_{23}\sigma_{21})$ is a product of $16$ elementary $\sigma$-conjugates. In Step 2 we show that $T_{kl}(x\bar\sigma_{23}\sigma_{2,-1})$ is a product of $16$ elementary $\sigma$-conjugates. In Step 3 we show that $T_{kl}(x\bar\sigma_{23}\sigma_{22})$ is a product of $32$ elementary $\sigma$-conjugates. In Step 4 we use Steps 1-3 in order to prove (i).\\
\\
{\bf Step 1.} Set $\tau:=T_{1,-2}(\bar 1^{-1}\bar\sigma_{23}\sigma_{23})T_{3,-2}(-\bar 1^{-1}\bar\sigma_{23}\sigma_{21})T_{3,-1}(\bar 1^{-2}\bar\sigma_{23}\sigma_{22})T_{3}(0,\bar\sigma_{22}\sigma_{21}+\overline{\bar\sigma_{22}\sigma_{21}})$ and $\xi:={}^{\sigma}\widetilde\tau$. One checks easily that $(\sigma\widetilde\tau v)_2=(\sigma v)_2$ for any $v\in V$. Further $\widetilde\tau\widetilde\sigma e_{-2}=\widetilde\sigma e_{-2}$. It follows that $(\xi v)_2=v_2$ for any $v\in V$ and further $\xi e_{-2}=e_{-2}$ (note that this implies that $(\widetilde\xi  v)_2=v_2$ for any $v\in V$ and further $\widetilde\xi  e_{-2}=e_{-2}$). Set  
\begin{align*}
\zeta:={}^{\widetilde\tau}[T_{-2,-1}(1),[\tau,\sigma]]={}^{\widetilde\tau}[T_{-2,-1}(1),\tau\xi]=[\widetilde\tau,T_{-2,-1}(1)][T_{-2,-1}(1),\xi],
\end{align*}
the last equality by Lemma \ref{pre}. It follows from the relations in Lemma \ref{39} that 
\[[\widetilde\tau,T_{-2,-1}(1)]=T_{3,-1}(\bar 1^{-1}\bar\sigma_{23}\sigma_{21})T_{1}(\alpha)\]
for some $\alpha\in \K_0$. On the other hand one checks easily that
\[[T_{-2,-1}(1),\xi]=T_{-2}(\beta)T_{12}(\xi_{11}-1)T_{32}(\xi_{31})T_{-3,2}(\xi_{-3,1})T_{-1,2}(\xi_{-1,1})\]
where $\beta=Q^0(\xi e_1)\hp (-1)\minus (0,\xi_{-2,1}+\bar\xi_{-2,1})$. Hence \[\zeta=T_{3,-1}(\bar 1^{-1}\bar\sigma_{23}\sigma_{21})T_{1}(\alpha)T_{-2}(\beta)T_{12}(\xi_{11}-1)T_{32}(\xi_{31})T_{-3,2}(\xi_{-3,1})T_{-1,2}(\xi_{-1,1}).\]
It follows from the relations in Lemma \ref{39} that $[T_{-1,3}(-\bar 1 x),[T_{-2,3}(1),\zeta]]=T_{-2,3}(x\bar\sigma_{23}\sigma_{21})$ for any $x\in R$. Hence we have shown
\[[T_{-1,3}(-\bar 1x),[T_{-2,3}(1),{}^{\widetilde\tau}[T_{-2,-1}(1),[\tau,\sigma]]]]=T_{-2,3}(x\bar\sigma_{23}\sigma_{21}).\]
This implies that $T_{-2,3}(x\bar\sigma_{23}\sigma_{21})$ is a product of $16$ elementary $\sigma$-conjugates. It follows from Lemma \ref{43} that $T_{kl}(x\bar\sigma_{23}\sigma_{21})$ is a product of $16$ elementary $\sigma$-conjugates.\\
\\
{\bf Step 2.} Step 2 can be done similarly. Namely one can show that 
\[[T_{-1,3}(-\bar 1 x),[T_{12}(1),{}^{\widetilde\tau}[T_{-2,1}(1),[\tau,\sigma]]]]=T_{-1,2}(x\bar\sigma_{23}\sigma_{2,-1})\]
for any $x\in R$ where $\tau=T_{21}(\bar 1^{-1}\bar\sigma_{23}\sigma_{23})T_{31}(-\bar 1^{-1}\bar\sigma_{23}\sigma_{22})T_{3,-2}(\bar 1^{-1}\bar\sigma_{23}\sigma_{2,-1})T_{3}(0,-\bar\sigma_{22}\sigma_{2,-1}-\overline{\bar\sigma_{22}\sigma_{2,-1}})$.\\
\\
{\bf Step 3.} Set $\tau:=T_{21}(-\bar 1^{-1}\bar\sigma_{22}\sigma_{23})T_{31}(\bar 1^{-1}\bar\sigma_{22}\sigma_{22})T_{2,-3}(\bar 1^{-1}\bar\sigma_{22}\sigma_{2,-1})T_{2}(0,-\bar\sigma_{23}\sigma_{2,-1}-\overline{\bar\sigma_{23}\sigma_{2,-1}})$ and $\xi:={}^{\sigma}\widetilde\tau$. One checks easily that $(\sigma\widetilde\tau v)_2=(\sigma v)_2$ for any $v\in V$. Further $\widetilde\tau\widetilde\sigma e_{-2}=\widetilde\sigma e_{-2}$. It follows that $(\xi v)_2=v_2$ for any $v\in V$ and further $\xi e_{-2}=e_{-2}$. Set  
\begin{align*}
\zeta:={}^{\widetilde\tau}[T_{32}(1),[\tau,\sigma]]={}^{\widetilde\tau}[T_{32}(1),\tau\xi]=[\widetilde\tau,T_{32}(1)][T_{32}(1),\xi].
\end{align*}
It follows from the relations in Lemma \ref{39} that $\eta:=[\widetilde\tau,T_{32}(1)]=T_{31}(-\bar 1^{-1}\bar\sigma_{22}\sigma_{23})T_{3}(\alpha)T_{3,-2}(a)$ for some $\alpha\in\K_0$ and $a\in R$. On the other hand one checks easily that $\theta:=[T_{32}(1),\xi]=T_{-2}(\beta)\prod\limits_{i\neq \pm 2}T_{i2}(x_i)$ for some $\beta\in\K_0$ and $x_1, x_3,x_{-3}, x_{-1}\in R$. Set 
\[\chi:={}^{\widetilde\eta }[T_{12}(1),\zeta]={}^{\widetilde\eta }[T_{12}(1),\eta\theta]=[\widetilde\eta ,T_{12}(1)][T_{12}(1),\theta].\]
It follows from the relations in Lemma \ref{39} that $[\widetilde\eta ,T_{12}(1)]=T_{32}(\bar 1^{-1}\bar\sigma_{22}\sigma_{23})T_{3}(\gamma)T_{3,-1}(b)$ for some $\gamma\in\K_0$ and $b\in R$ and $[T_{12}(1),\theta]=T_{-2}(\delta)$ for some $\delta\in \K_0$. Hence $\chi=T_{32}(\bar 1^{-1}\bar\sigma_{22}\sigma_{23})T_{3}(\gamma)T_{3,-1}(b)T_{-2}(\delta)$. Therefore
\[[T_{-2,3}(\overline{(\bar 1)^3}\bar x),[T_{2,-1}(1),\chi]]=T_{-2,-1}(-\overline{(\bar 1)^3}\bar 1^{-1}\bar x\bar\sigma_{22}\sigma_{23})=T_{12}(x\bar\sigma_{23}\sigma_{22})\]
for any $x\in R$, again by the relations in Lemma \ref{39}. Hence we have shown
\[[T_{-2,3}(\overline{(\bar 1)^3}\bar x),[T_{2,-1}(1),{}^{\widetilde\eta }[T_{12}(1),{}^{\widetilde\tau}[T_{32}(1),[\tau,\sigma]]]]]=T_{12}(x\bar\sigma_{23}\sigma_{22}).\]
This implies that $T_{12}(x\bar\sigma_{23}\sigma_{22})$ is a product of $32$ elementary $\sigma$-conjugates. It follows from Lemma \ref{43} that $T_{kl}(x\bar\sigma_{23}\sigma_{22})$ is a product of $32$ elementary $\sigma$-conjugates.\\
\\
{\bf Step 4.} Set $I:=I(\bar\sigma_{23}\sigma_{21},\overline{\bar\sigma_{23}\sigma_{2,-1}})$, $J:=I(\bar\sigma_{23}\sigma_{21},\overline{\bar\sigma_{23}\sigma_{2,-1}},\bar\sigma_{23}\sigma_{22})$ and
$\tau:=[\widetilde\sigma ,T_{12}(-\bar\sigma_{23})]$. One checks easily that
\begin{align*}
\tau_{11}&=1-\widetilde\sigma_{11}\bar\sigma_{23}\sigma_{21}+\bar 1^{-1}\widetilde\sigma_{1,-2}\sigma_{23}\sigma_{-1,1}\\
&=1-\widetilde\sigma_{11}\underbrace{\bar\sigma_{23}\sigma_{21}}_{\in I}-\bar 1^{-3}\underbrace{\bar\sigma_{2,-1}\sigma_{23}}_{\in I}\sigma_{-1,1},
\end{align*}
the last equality by Proposition \ref{propodd}, and 
\begin{align*}
\tau_{12}&=-\widetilde\sigma_{11}\bar\sigma_{23}\sigma_{22}+\bar 1^{-1}\widetilde\sigma_{1,-2}\sigma_{23}\sigma_{-1,2}+\tau_{11}\bar\sigma_{23}\\
&=-\widetilde\sigma_{11}\underbrace{\bar\sigma_{23}\sigma_{22}}_{\in J}-\bar 1^{-3}\underbrace{\bar\sigma_{2,-1}\sigma_{23}}_{\in J}\sigma_{-1,2}+\tau_{11}\bar\sigma_{23},
\end{align*}
again the last equality by Proposition \ref{propodd}.
Hence $\tau_{11}\equiv 1 \bmod I$ and $\tau_{12}\equiv \bar\sigma_{23}\bmod J$ (note that $I\subseteq J$). Set 
$\zeta:={}^{P_{13}P_{21}}\tau$. Then $\zeta_{22}=\tau_{11}$ and $\zeta_{23}=\tau_{12}$ and hence $\bar\zeta_{23}\zeta_{22}\equiv \sigma_{23}\bmod I+\bar J$. Applying Step 3 above to $\zeta$, we get that $T_{kl}(\bar\zeta_{23}\zeta_{22})$ is a product of $32$ elementary $\zeta$-conjugates. Since any elementary $\zeta$-conjugate is a product of $2$ elementary $\sigma$-conjugates, it follows that $T_{kl}(\bar\zeta_{23}\zeta_{22})$ is a product of $64$ elementary $\sigma$-conjugates. Thus, by Steps 1-3, $T_{kl}(\sigma_{23})$ is a product of $64+16+16+16+16+32=160$ elementary $\sigma$-conjugates. Since one can bring $\sigma_{ij}$ to position $(3,2)$ by conjugating elementary permutations, the assertion of (i) follows.\\
\\
(ii) One checks easily that $({}^{T_{ji}(1)}\sigma)_{j,-i}=\sigma_{i,-i}+\sigma_{j,-i}$. Applying (i) to $^{T_{ji}(1)}\sigma$ we get that $T_{kl}(\sigma_{i,-i}+\sigma_{j,-i})$ is a product of $160$ elementary $\sigma$-conjugates (note that any elementary $^{T_{ji}(1)}\sigma$-conjugate is also an elementary $\sigma$-conjugate). Applying (i) to $\sigma$ we get that $T_{kl}(\sigma_{j,-i})$ is a product of $160$ elementary $\sigma$-conjugates. It follows that $T_{kl}(\sigma_{i,-i})=T_{kl}(\sigma_{i,-i}+\sigma_{j,-i})T_{ji}(-\sigma_{j,-i})$ is a product of $320$ elementary $\sigma$-conjugates.\\
\\
(iii) Choose an $x\in R$ such that $(v_0,x)\in \K_0$. One checks easily that $({}^{T_{-j}(v_0,x)}\sigma)_{ij}=\sigma_{ij}-\sigma_{i,-j}\epsilon_{-j}\bar x+(\sigma^{0h}v_0)_i$. Applying (i) to $^{T_{-j}(v_0,x)}\sigma$ we get that $T_{kl}(\sigma_{ij}-\sigma_{i,-j}\epsilon_{-j}\bar x+(\sigma^{0h}v_0)_i)$ is a product of $160$ elementary $\sigma$-conjugates. Applying (i) to $\sigma$ we get that $T_{kl}(-\sigma_{ij})$ and $T_{kl}(\sigma_{i,-j}\epsilon_{-j}\bar x)$ each are a product of $160$ elementary $\sigma$-conjugates. It follows that $T_{kl}((\sigma^{0h}v_0)_i)=T_{kl}(\sigma_{ij}-\sigma_{i,-j}\epsilon_{-j}\bar x+(\sigma^{0h}v_0)_i)T_{kl}(-\sigma_{ij})T_{kl}(\sigma_{i,-j}\epsilon_{-j}\bar x)$ is a product of $480$ elementary $\sigma$-conjugates.\\
\\
(iv) Choose an $x\in R$ such that $(v_0,x)\in \K_0$. One checks easily that $({}^{T_{i}(v_0,x)}\sigma)_{ij}=-\epsilon_iB_0(v_0,\sigma^{h0}e_j)+\sigma_{ij}+\epsilon_ix\sigma_{-i,j}$. Applying (i) to $^{T_{i}(-1)}\sigma$ we get that $T_{kl}(B_0(v_0,\sigma^{h0}e_j)-\epsilon_i^{-1}\sigma_{ij}-x\sigma_{-i,j})$ is a product of $160$ elementary $\sigma$-conjugates. Applying (i) to $\sigma$ we get that $T_{kl}(\epsilon_i^{-1}\sigma_{ij})$ and $T_{kl}(x\sigma_{-i,j})$ each are a product of $160$ elementary $\sigma$-conjugates. It follows that $T_{kl}(B_0(v_0,\sigma^{h0}e_j))=T_{kl}(B_0(v_0,\sigma^{h0}e_j)-\epsilon_i^{-1}\sigma_{ij}-x\sigma_{-i,j})T_{kl}(\epsilon_i^{-1}\sigma_{ij})T_{kl}(x\sigma_{-i,j})$ is a product of $480$ elementary $\sigma$-conjugates.\\
\\
(v) One checks easily that $({}^{T_{ji}(1)}\sigma)_{ji}$ equals $\sigma_{ii}-\sigma_{jj}+\sigma_{ji}-\sigma_{ij}$. Applying (i) to $^{T_{ji}(1)}\sigma$ we get that $T_{kl}(\sigma_{ii}-\sigma_{jj}+\sigma_{ji}-\sigma_{ij})$ is a product of $160$ elementary $\sigma$-conjugates. Applying (i) to $\sigma$ we get that $T_{kl}(\sigma_{ij})$ and $T_{kl}(-\sigma_{ji})$ each are a product of $160$ elementary $\sigma$-conjugates. It follows that $T_{kl}(\sigma_{ii}-\sigma_{jj})=T_{kl}(\sigma_{ii}-\sigma_{jj}+\sigma_{ji}-\sigma_{ij})T_{kl}(\sigma_{ij})T_{kl}(-\sigma_{ji})$ is a product of $480$ elementary $\sigma$-conjugates.\\
\\
(vi) Follows from (v) since $T_{kl}(\sigma_{ii}-\sigma_{-i,-i})=T_{kl}(\sigma_{ii}-\sigma_{jj})T_{kl}(\sigma_{jj}-\sigma_{-i,-i})$.\\
\\
(vii) Set $m:=160$. In Step 1 we show that if $x\in R$, then $T_{k}(Q^0(\sigma e_1)\hp\sigma_{11}x)$ is a product of $25m+4$ elementary $\sigma$-conjugates. In Step 2 we use Step 1 in order to prove (vii).\\
\\
{\bf Step 1.} Set $u:=e_{-2}\widetilde\sigma_{-1,-1}-e_{-1}\widetilde\sigma_{-1,-2}\in V_h$ and $v:=\widetilde\sigma u$. Then clearly $v_{-1}=0$. Further $v$ is isotropic since $u$ is isotropic and $\widetilde\sigma $ is unitary. By Lemma \ref{lemesdspec} we have $T_{e_1v}(0)\in \EU_6(\V_0)$. By Lemma \ref{lemesd}, $\widetilde{T_{e_1v}(0)}=T_{e_1,-v}(0)$ since $B(v,v)=B(\widetilde\sigma u,\widetilde\sigma u)=B(u,u)=0$. Set $\xi:={}^{\sigma}T_{e_1,-v}(0)=T_{\sigma e_1,-u}(0)$. One checks easily that
\begin{align*}
\xi^{hh}=
\arraycolsep=4pt\def\arraystretch{1.5}\left(\begin{array}{ccc|ccc}
1-\sigma_{11}\sigma_{21}&\sigma_{11}\sigma_{11}&0&0&0&0\\
-\sigma_{21}\sigma_{21}&1+\sigma_{21}\sigma_{11}&0&0&0&0\\
-\sigma_{31}\sigma_{21}&\sigma_{31}\sigma_{11}&1&0&0&0\\
\hline-\sigma_{-3,1}\sigma_{21}&\sigma_{-3,1}\sigma_{11}&0&1&0&0\\
\bar\beta&\alpha&\overline{\sigma_{-3,1}\sigma_{11}}&-\bar 1^{-1}\overline{\sigma_{31}\sigma_{11}}&1-\bar 1^{-1}\overline{\sigma_{21}\sigma_{11}}&-\bar 1^{-1}\overline{\sigma_{11}\sigma_{11}}\\
\gamma&\beta&-\overline{\sigma_{-3,1}\sigma_{21}}&\bar 1^{-1}\overline{\sigma_{31}\sigma_{21}}&\bar 1^{-1}\overline{\sigma_{21}\sigma_{21}}&1+\bar 1^{-1}\overline{\sigma_{11}\sigma_{21}}
\end{array}\right)
\end{align*}
where $\alpha=\sigma_{-2,1}\sigma_{11}+\overline{\sigma_{-2,1}\sigma_{11}}$, $\beta=\sigma_{-1,1}\sigma_{11}-\overline{\sigma_{-2,1}\sigma_{11}}$ and $\gamma=-\sigma_{-1,1}\sigma_{21}-\overline{\sigma_{-1,1}\sigma_{21}}$. Further
\[\xi^{h0}(e_j)=
\begin{cases}
-\sigma^{h0}e_1\sigma_{21},&\text{ if }j=1,\\
\sigma^{h0}e_1\sigma_{11},&\text{ if }j=2,\\
0,&\text{ if }j\neq 1,2,
\end{cases}, 
\quad \xi^{0h}(w_0)=-uB_0(\sigma^{h0}e_1, w_0)~\text{ and }~\xi^{00}=id_{V_0}.\]
Set \[\tau:=T_{-3,1}(\sigma_{-3,1}\sigma_{21})T_{-3,2}(-\sigma_{-3,1}\sigma_{11}).\] 
It follows from (i) that $\tau$ is a product of $2m$ elementary $\sigma$-conjugates. Then
\begin{align*}
(\xi\tau)^{hh}=
\arraycolsep=10pt\def\arraystretch{1.5}\left(\begin{array}{ccc|ccc}
*&\sigma_{11}\sigma_{11}&0&0&0&0\\
*&1+\sigma_{21}\sigma_{11}&0&0&0&0\\
*&\sigma_{31}\sigma_{11}&1&0&0&0\\
\hline0&0&0&1&0&0\\
*&\alpha+\delta&0&*&*&*\\
*&\beta-\epsilon&0&*&*&*
\end{array}\right)
\end{align*}
where $\delta=\bar 1^{-1}\overline{\sigma_{31}\sigma_{11}}\sigma_{-3,1}\sigma_{11}$, $\epsilon=\bar 1^{-1}\overline{\sigma_{31}\sigma_{21}}\sigma_{-3,1}\sigma_{11}$ and a $*$ stands for an entry we are not interested in very much. Further $(\xi\tau)^{h0}=\xi^{h0}$, $(\xi\tau)^{0h}=\xi^{0h}$ and $(\xi\tau)^{00}=\xi^{00}$. Let $x\in R$ and set  
\begin{align*}
\zeta:=&{}^{T_{e_1,-v}(0)}[T_{2,-3}(x),[T_{e_1v}(0),\sigma]\tau]\\
=&{}^{T_{e_1,-v}(0)}[T_{2,-3}(x),T_{e_1v}(0)\xi\tau]\\
=&[T_{e_1,-v}(0),T_{2,-3}(x)][T_{2,-3}(x),\xi\tau].
\end{align*}
Clearly $\zeta$ is a product of $4m+4$ elementary $\sigma$-conjugates. It follows from the relations in Lemma \ref{39} that
\begin{align*}
&[T_{e_1,-v}(0),T_{2,-3}(x)]\\
=&T_{1}(0,x\bar v_{-2}v_{-3}+\overline{x\bar v_{-2}v_{-3}}))T_{2,-1}(xv_{-3})T_{1,-3}(\bar 1^{-1} x\bar v_{-2})\\
=&T_{1}(0,a+\bar a)T_{1,-2}(b)T_{1,-3}(x(\sigma_{11}\sigma_{22}-\sigma_{12}\sigma_{21}))
\end{align*}
for some $a,b\in I(\sigma_{21},\sigma_{23})$. With a little effort one can show that
\[[T_{2,-3}(x),\xi\tau]=T_{1,-3}(-x\sigma_{11}\sigma_{11})T_{2,-3}(-x\sigma_{21}\sigma_{11})T_{-2,-3}(-x(\alpha+\delta))T_{-1,-3}(-x(\beta-\epsilon))T_{3}(\chi)\]
where $\chi=Q^0(\sigma e_1)\hp\sigma_{11}x\+(0,c+\bar c)$ for some $c\in I(\sigma_{31},\sigma_{-2,1})$. Hence
\begin{align*}
\zeta=&T_{1}(0,a+\bar a)T_{1,-2}(b)T_{1,-3}(x(\sigma_{11}(\sigma_{22}-\sigma_{11})-\sigma_{12}\sigma_{21}))\cdot\\
&\cdot T_{2,-3}(-x\sigma_{21}\sigma_{11})T_{-2,-3}(-x(\alpha+\delta))T_{-1,-3}(-x(\beta-\epsilon))T_{3}(\chi).
\end{align*}
It follows from (i), (ii), (v) and relation (R9) in Lemma \ref{39} that $T_{3}(\chi)$ is a product of $4m+4+4m+2m+4m+m+3m+3m=21m+4$ elementary $\sigma$-conjugates. By (i) and relation (S5) in Lemma \ref{39}, $T_{3}(0,-c-\bar c)$ is a product of $4m$ elementary $\sigma$-conjugates. Hence $T_{3}(Q^0(\sigma e_1)\hp\sigma_{11}x)=T_{3}(\chi)T_{3}(0,-c-\bar c)$ is a product of $25m+4$ elementary $\sigma$-conjugates. It follows from Lemma \ref{43} that $T_{k}(Q^0(\sigma e_1)\hp\sigma_{11}x)$ is a product of $25m+4$ elementary $\sigma$-conjugates.\\
\\
{\bf Step 2.} Clearly 
\[1=(\widetilde\sigma\sigma e_1)_1=(\widetilde\sigma ^{hh}\sigma ^{hh}e_1)_1+\underbrace{(\widetilde\sigma ^{0h}\sigma^{h0}e_1)_1}_{x:=}=
\sum\limits_{s\in\Theta}\widetilde\sigma _{1s}\sigma_{s1}+x.\]
It follows from Lemma \ref{last} that
\begin{align*}
&T_{k}(Q^0(\sigma e_1))\\
=&T_{k}(Q^0(\sigma e_1)\hp(\sum\limits_{s\in\Theta}\widetilde\sigma _{1s}\sigma_{s1}+x))\\
=&T_{k}(\underbrace{\plus\limits_{s=1}^{-1} Q^0(\sigma e_1)\hp\widetilde\sigma _{1s}\sigma_{s1}}_{A:=}\+\underbrace{Q^0(\sigma e_1)\hp x}_{B:=}\+\underbrace{(0,y+\bar y)}_{C:=})\\
=&T_k(A)T_k(B)T_k(C)
\end{align*}
for some $y\in I(F_h(\sigma e_1))$ (note that $Q^0(\sigma e_1)\in \K_0$ by Proposition \ref{propodd}). By Step 1, $T_{k}(Q^0(\sigma e_1)\hp\widetilde\sigma _{11}\sigma_{11})$ is a product of $25m+4$ elementary $\sigma$-conjugates. By (i) and relation (R8) in Lemma \ref{39}, $T_k(Q^0(\sigma e_1)\hp\widetilde\sigma _{1s}\sigma_{s1})$ is a product of $3m$ elementary $\sigma$-conjugates if $s\neq \pm 1$. By (ii) and relation (R8), $T_k(Q^0(\sigma e_1)\hp\widetilde\sigma _{1,-1}\sigma_{-1,1})$ is a product of $6m$ elementary $\sigma$-conjugates. Hence $T_k(A)$ is a product of $43m+4$ elementary $\sigma$-conjugates. By (iv) and relation (R8), $T_k(B)$ is a product of $9m$ elementary $\sigma$-conjugates (note that $\sigma^{h0}e_1\in V^{ev}_0$). Since $F_h(\sigma e_1)=-\sum\limits_{r=1}^3\bar\sigma_{r1}\bar 1^{-1}\sigma_{-r,1}$, it follows from (i), (ii) and relation (R9) in Lemma \ref{39} that $T_k(C)$ is a product of $4m+2m+2m=8m$ elementary $\sigma$-conjugates. Hence $T_{k}(Q^0(\sigma e_1))$ is a product of $60m+4=9604$ elementary $\sigma$-conjugates. The assertion of (vii) follows now from Lemma \ref{new}.\\
\\
(viii) Choose an $x\in R$ such that $(v_0,x)\in \K_0$. Set $\xi:={}^{T_{-j}(v_0,x)}\sigma$. One checks easily that \[\xi^{h0}e_j=\sigma^{00}v_0-v_0\sigma_{jj}+\sigma^{h0}e_j-\sigma^{h0}e_{-j}\epsilon_{-j}\bar x+v_0\sigma_{j,-j}\epsilon_{-j}\bar x-v_0(\sigma^{0h}v_0)_j.\]
By applying (vii) to $\xi$ we get that $T_{k}(Q^0(\xi e_j))=T_{k}(\xi^{h0}e_j, a)$ is a product of $10564$ elementary $\sigma$-conjugates where $a$ is some ring element. By applying (vii) to $\sigma$ we get that $T_{k}(Q^0(\sigma e_j)\hp(-1))=T_k(-\sigma^{h0} e_j,b)$ and $T_{k}(Q^0(\sigma e_{-j})\hp \epsilon_{-j}\bar x)=T_{k}(\sigma^{h0}e_{-j}\epsilon_{-j}\bar x,c)$ each are a product of $10564$ elementary $\sigma$-conjugates where $b$ and $c$ are some ring elements. By (ii) and relation (R8) in Lemma \ref{39}, $T_{k}(-v_0\sigma_{j,-j}\epsilon_{-j}\bar x,d)$ is a product of $3\cdot 320=960$ elementary $\sigma$-conjugates where $d$ is some ring element. By (iii) and relation (R8) in Lemma \ref{39}, $T_{k}(v_0(\sigma^{0h}v_0)_j,f)$ is a product of $3\cdot 480=1440$ elementary $\sigma$-conjugates where $f$ is some ring element. It follows that 
\[T_{k}(\xi^{h0}e_j, a)T_k(-\sigma^{h0} e_j,b)T_{k}(\sigma^{h0}e_{-j}\epsilon_{-j}\bar x,c)T_{k}(-v_0\sigma_{j,-j}\epsilon_{-j}\bar x,d)T_{k}(v_0(\sigma^{0h}v_0)_j,f)=T_k(\sigma^{00}v_0-v_0\sigma_{jj},x')\] is a product of $3\cdot 10564+960+1440=34092$ elementary $\sigma$-conjugates where $x'$ is some ring element.
\end{proof}

We can now deduce our main result from Theorems \ref{thmscf} and \ref{thmm}.
\begin{SCT}\label{thmsct}
Let $H$ be a subgroup of $\U_{6}(\V_0)$. Then $H$ is E-normal iff 
\begin{equation}
\EU_{6}(\V_0,\I_0)\subseteq H\subseteq \CU_{6}(\V_0,\I_0)
\end{equation}
for some odd form ideal $\I_0$. Furthermore, $\I_0$ is uniquely determined, namely it is the level of $H$.
\end{SCT}
\begin{proof}
First suppose that $H$ is E-normal. Let $\I_0$ denote the level of $H$. Then $\EU_{6}(\V_0,\I_0)\subseteq H$. It remains to show that $H\subseteq \CU_{6}(\V_0,\I_0)$. Let $\sigma\in H$. By Corollary \ref{cornu} we have $\sigma\in \NU_{6}(\V_0,\I_0)$. Theorem \ref{thmm} and Proposition \ref{propcumax} imply that $\sigma\in \CU_{6}(\V_0,(\I_0)_{\max})$. Further Theorem \ref{thmm} implies that $Q^0(\sigma e_i)\in\M_0$ for any $i\in\Theta$ and $Q^0([\sigma,\tau] e_i)\in\M_0$ for any elementary transvection $\tau$ and $i\in\Theta$. It follows from Lemma \ref{lemcu} that $\sigma\in \CU_{6}(\V_0,\I_0)$. Thus we have shown that (6) holds. Conversely, suppose that (6) holds for some odd form ideal $\I_0$. Then it follows from the standard commutator formula in Theorem \ref{thmscf} that $H$ is E-normal. The uniqueness of $\I_0$ follows from the standard commutator formula and the easy-to-check fact that $\EU_{6}(\V_0,\I_0)\subseteq \U_{6}(\V_0,\J_0)~\Leftrightarrow~\I_0\subseteq \J_0$.
\end{proof}

\bibliographystyle{plain} 
\bibliography{sct_petrov}

\end{document}